\documentclass[letterpaper, 10 pt, conference]{ieeeconf} 
\IEEEoverridecommandlockouts                              

\usepackage{cite}
\usepackage{amsmath,amssymb,amsfonts}
\usepackage{algpseudocode}
\usepackage{graphicx}
\usepackage{algorithm}
\usepackage{textcomp}
\usepackage{xcolor}
\usepackage{bm}
\usepackage{multirow}
\usepackage[abs]{overpic}
\usepackage{geometry}
\newcommand{\inner}[1]{\langle #1 \rangle}

\newcommand{\calM}{\mathcal{M}}
\newcommand{\calN}{\mathcal{N}}

\newcommand{\calQ}{\mathcal{Q}}

\newcommand{\calS}{\mathcal{S}}

\newcommand{\calU}{\mathcal{U}}

\newcommand{\bbR}{\mathbb{R}}

\newcommand{\bfA}{\mathbf{A}}
\newcommand{\bfB}{\mathbf{B}}

\newcommand{\bfD}{\mathbf{D}}

\newcommand{\bfQ}{\mathbf{Q}}

\newcommand{\bfS}{\mathbf{S}}

\newcommand{\bfs}{\mathbf{s}}
\newcommand{\bfu}{\mathbf{u}}
\newcommand{\bfv}{\mathbf{v}}

\newcommand{\bfx}{\mathbf{x}}

\newcommand{\diff}{\text{d}}

\newcommand{\mat}[1]{\begin{bmatrix} #1 \end{bmatrix}}

\newtheorem{theorem}{Theorem}[section]

\newtheorem{lemma}[theorem]{Lemma}

\newtheorem{prop}[theorem]{Proposition}

\title{\LARGE \bf
{Intrinsic} Successive Convexification:\\ Trajectory Optimization on Smooth Manifolds
}

\author{Spencer Kraisler$^{1}$, Mehran Mesbahi$^{2}$, and Behcet Açikmese$^{3}$
\thanks{$^{1}$The authors are with the William E. Boeing Department of Aeronautics and Astronautics, University of Washington, Seattle, WA, U.S.A {\tt\small kraisler+mesbahi+behcet@uw.edu}}%
}
\begin{document}

\maketitle
\thispagestyle{empty}
\pagestyle{empty}

\begin{abstract}
A fundamental issue at the core of trajectory optimization on smooth manifolds is handling the implicit manifold constraint within the dynamics. The conventional approach is to enforce the dynamic model as a constraint. However, we show this approach leads to significantly redundant operations, as well as being heavily dependent on the state space representation. Specifically, we propose an \text{\em{intrinsic}} successive convexification methodology for optimal control on smooth manifolds. This so-called iSCvx is then applied to a representative example involving attitude trajectory optimization for a spacecraft subject to non-convex constraints.
\end{abstract}
\section{Introduction}

A fundamental challenge in formulating and solving trajectory optimization problems on manifolds lies in selecting the most appropriate state space representation. For rigid bodies, common representations include unit dual quaternions and homogeneous matrices. The choice of representation directly influences the solution strategy and the performance of algorithms for solving the corresponding non-linear optimal control problem. For example, the widespread adoption of dual quaternions is due to their advantageous properties in this context \cite{lee2015optimal, lee2012dual, szmuk2019successive}. 
In this paper, we investigate a trajectory optimization framework that is invariant to representation, provided the dynamics and constraints depend exclusively on coordinate-free quantities of the system manifolds. We frame this perspective in the context of the Successive Convexification (SCvx) algorithm, a popular method for fast trajectory optimization~\cite{mao2018successive,malyuta2022convex}.

The extrinsic enforcement of dynamics leads to certain redundant operations, particularly in SCvx, by searching over excess dimensions in the convex sub-problem. A systematic analysis of the extrinsic approach is done in \cite{bonalli2019trajectory}. For example, the unit quaternions $\calQ$ is a 3-dimension manifold embedded in a 4-dimensional ambient vector space $\bbR^4$. The fact that the linearized dynamics of the vehicle are parameterized by matrices $A_i^k \in \bbR^{4 \times 4}$ and $B_i^k \in \bbR^{4 \times 3}$ imply that we have redundancy in our parameterization. The same observation goes for the so-called virtual control $v_i \in \bbR^4$ that appears in SCvx. This observation implies that our optimizer requires more variables than necessary, resulting in more memory usage and higher time complexity. 

Furthermore, the virtual control term introduced in the SCvx setup is essentially responsible for achieving two distinct objectives: satisfying the dynamical as well as the state-control constraints, while also trying to meet the implicit manifold constraint (i.e., $q_i^{\top} q_i = 1$). This observation highlights that if there was a way to intrinsically embed the SCvx procedure into the system manifold, then the virtual control would only need to deal with enforcing the dynamical and state-control constraints. There would also be fewer ``search directions'' when solving the resulting local sub-problems. Moreover, since iterates may not remain on the system manifold, additional iterates are required to converge. And most importantly, SCvx is highly dependent on how the dynamics are embedded in the ambient manifold. 

The intrinsic approach to general constrained optimization over manifolds is an established topic with especially recent interest \cite{smith2014optimization, liu2020simple, hu2020brief, bergmann2019intrinsic}. For intrinsic \textit{trajectory} optimization, the authors in \cite{saccon2013optimal, hauser1998trajectory} optimize over the Banach manifold of system trajectories defined on Lie groups.

The contribution of this paper is developing a SCvx-like procedure that is invariant to representation we call intrinsic SCvx (iSCvx). We will show that the proposed approach resolves the redundancy and representation-dependent issues mentioned above. We focus on discrete-time systems with finite horizon. The outline of the paper is as follows. In \S\ref{problem-statment} we present the problem statement, followed by an overview of required background on smooth manifolds and SCvx in \S\ref{background}. The main results of the paper are then discussed in \S\ref{iSCvx} along with theoretical guarantees on the solution of the local sub-problem; this is then followed by an example of the proposed iSCvx algorithm to constrained attitude control in \S\ref{sec:example}. We conclude with a discussion of future directions of this line of work in \S\ref{conclusion}.

\section{Problem Statement} \label{problem-statment}
Let $\calM$ and $\calU$ be Riemannian manifolds and $f \colon\calM \times \calU \to \calM$ represent the system dynamics. The stage and final stage costs are denoted $\phi:\calM \times \calU \to \bbR$ and $h:\calM \to \bbR$. In this direction, define the trajectory cost as,
\begin{equation}\label{eq:traj-cost}
    C(\bfx, \bfu) := \sum_{i=0}^{N-1} \phi(x_i,u_i) + h(x_N);
\end{equation}
here, $N > 0$ designates the time horizon of the problem, $\bfx=(x_0,\ldots,x_N) \in \calM^{N + 1}$ is the state trajectory, and $\bfu = (u_0,\ldots,u_{N-1}) \in \calU^N$ is the control trajectory.

We define the constraints $s:\calM \times \calU \to \bbR^{n_s}$; we assume that all functions defined above are twice continuously differentiable. The trajectory optimization problem of interest 
\newgeometry{top=57pt, left=48pt, right=48pt, bottom=43pt}
is solving
\begin{subequations}\label{eq:OCP}
\begin{align}
    \min_{(\bfx, \bfu)} \ &C(\bfx, \bfu) \\
    \text{s.t. } &x_{i+1} = f(x_i, u_i), \; i=1,2,\ldots, N, \label{eq:dynamics}\\
    &s(x_i, u_i) \leq 0, \; i=1,2,\ldots, N, \label{eq:constraints} \\
    &x_0 \text{ given}.
\end{align}
\end{subequations}
The key idea pursued in this work is solving (\ref{eq:OCP}) in an ``intrinsic'' manner--one that is invariant with respect to how $\calM$ has been embedded in the ambient vector space.

\section{Background} \label{background}

A smooth manifold $\calM$ is a topological space that locally behaves like a vector space \cite{boumal2023introduction,lee2012smooth, lee2018introduction}. Euclidean submanifolds are subsets of $\calM \subset \bbR^{d}$ with no corners, cusps, or ``non-smooth'' artifacts, such as the unit quaternions $\calQ := \{q \in \bbR^4:q^{\top} q=1\} \subset \bbR^4$; see \S\ref{sec:example}.

A smooth curve is a smooth function $\gamma:\bbR \to \calM$. The tangent space $T_x \calM$ at $x$ is the set of the tangent vectors $\dot{\gamma}(0)$ of all smooth curves $\gamma(\cdot)$ with $\gamma(0)=x$. The tangent space is a vector space whose dimension equals the manifold dimension. For example, $T_q\calQ := \{v \in \bbR^4:v^{\top}q=0\}$ has dimension 3. The disjoint union of tangent spaces $T \calM$ is called the tangent bundle. 

A vector field is a smooth function $V:\calM \to T\calM$ with $V(x) \in T_x\calM$. Let $U \subset \calM$. A local frame $(E_i:U \to T\calM)$ is an ordered set of $\dim \calM$ linearly independent vector fields; global frames have $U=\calM$.

Let $f \colon \calM \to \calN$ be smooth. The differential $\diff f_x:T_x\calM \to T_{f(x)}\calN$ of $f$ at $x$ along $v \in T_x\calM$ is the linear mapping 
\begin{equation}\label{eq:diff}
    \diff f_x(v) := \frac{d}{dt}\Big|_{t=0} f \circ \gamma,
\end{equation} where $\gamma(\cdot)$ is any smooth curve satisfying $(\gamma(0), \dot{\gamma}(0))=(x,v)$. The differential is the canonical way of defining a directional derivative. The following example makes use of the unit quaternions. Fix $y_b \in \bbR^3$. Define
\begin{equation*}\label{eq:rotation}
    f(q) := q \cdot y_b \cdot q^{-1}, 
\end{equation*} that rotates $y_b$ via the rotation represented by $q$. Here, ``$\cdot$'' denotes quaternion multiplication, and we are identifying $\bbR^3$ with the pure quaternions. From the above, we have
\begin{equation*}
    \diff f_q(v) = q \cdot [y_b, q^{-1} \cdot v] \cdot q^{-1},
\end{equation*}  where $[a,b] := a \cdot b - b \cdot a$.

A retraction is a smooth mapping $R:\calS \subset T\calM \to \calM$ where $\calS$ is open, $(x,0_x) \in \calS$ for all $x$, and the curve $\gamma(t):=R_x(tv) \equiv R(x,tv)$ satisfies $(\gamma(0),\dot{\gamma}(0))=(x,v)$ for each $(x,v) \in \calS$ \cite{boumal2019global,truong2020unconstrained, lezcano2019trivializations}. For vector spaces, the canonical retraction is $R_x(v)=x + v$. For a product manifold $\calM^k$, the induced product retraction is the retraction applied element-wise: $R_{\bfx}(\bfv) = (R_{x_1}(v_1),\ldots,R_{x_k}(v_k))$. 

By the inverse function theorem, there is a neighborhood $U_x$ around each $0 \in T_x \calM$ such that $R_x$ admits an inverse: $R_x^{-1}:U_x \subset T_x \calM \to \calM$.

The canonical retraction for unit quaternions is given by 
\begin{subequations}
    \begin{align}
        R_q(v) & := q \cdot \exp(\omega), \label{eq:lie-retraction}\\
        \exp(\omega)& := (\cos\|\omega\|_2, \text{sinc}(\omega) \omega), \label{eq:lie-exp}
    \end{align}
\end{subequations} where $\text{sinc}(\omega) := \sin \|\omega\|_2/\|\omega\|_2$. Here, we are expressing quaternions as $q=(q_0,q_v)$, with $q_0 \in \bbR$ the real part and $q_v \in \bbR^3$ the pure part. Also, $v = q \cdot \omega \in T_q \calQ$ and $\exp:\mathfrak{q} \to \calQ$ is the quaternion exponential. The inverse of this retraction becomes,
\begin{subequations}
\begin{align*}
    R_q^{-1}(p) &:= q \cdot \log(q \cdot p^{-1}),\\
    \log(q) &:= \frac{\cos^{-1}(q_0)}{\|q_v\|_2}q_v,
\end{align*} 
\end{subequations} where $\log \colon \calQ \to \mathfrak{q}$ denotes the quaternion logarithm. 

\subsection{Successive Convexification (SCvx)}
SCvx is a primal solver of non-linear optimal control problems. At its $k$th iteration, we begin with a possibly infeasible trajectory $(\bfx^k, \bfu^k)$. Next, the dynamics and constraints of Problem \ref{eq:OCP} are linearized about this trajectory, resulting in linear constraints and dynamics of the form,
\begin{align*}
    x_{i+1}^k + \eta_{i+1} = f(x_i^k, u_i^k) + A_i^k \eta_i + B_i^k \xi_i, \\
    s(x_i^k,u_i^k) + S_i^k \eta_i + Q_i^k \xi_i \leq 0.
\end{align*} The convex sub-problem is then solved, and the resulting optimal perturbations are added to the current trajectory to obtain a more optimal and feasible trajectory. The process then repeats. To ensure that the convex sub-problem is feasible, relaxation parameters are added, such as the virtual control parameters to ensure reachability of the linearized dynamics and the virtual buffers for domain feasibility. A trust region radius constant is also added to bound the control perturbation. Penalties are appended to the sub-problem cost to ensure these relaxation parameters are small.

\section{Intrinsic SCvx} \label{iSCvx}
In this section, we describe the main mechanisms used in iSCvx. First, we show how to linearize dynamics in an intrinsic way in \S\ref{subsec:linearization}. Next, we show how to handle the virtual control and buffer zone terms, along with the penalty in \S\ref{subsec:slack-penalty}. Thereafter, we show how to compute an intrinsic 2nd-order approximation of the trajectory cost in \S\ref{subsec:linear-cost}. Then, in \S{\ref{subsec:geo-local-subproblem}}, we introduce the geodesic local sub-problem and how to numerically solve it. Last, we present the iSCvx algorithm in \S\ref{subsec:algo}.

\subsection{Linearizing Dynamics on a Smooth Manifold}\label{subsec:linearization}

The key difference between iSCvx and SCvx is we linearize dynamics using the intrinsic differential (\ref{eq:diff}) rather than the ordinary Jacobian operator. If $\calM \subset \bbR^{d_x}$ and $\calU \subset \bbR^{d_u}$, the resulting system matrices in the linearized dynamics are $d_x \times d_x$ and $d_x \times d_u$. Using the differential, the resulting system matrices will instead have the smaller dimensions $n \times n$ and $n \times m$, where $n=\dim\calM$ and $m=\dim \calU$.  

At the $k$th iteration of iSCvx, we have a possibly infeasible trajectory $(\bfx^k, \bfu^k) \subset \calM^{N + 1} \times \calU^N$. We will refer to the retractions of $\calM$ and $\calU$ both as $R$ for notational simplicity. 

We perturb $(\bfx^k, \bfu^k)$ with control perturbation $\bm{\xi}$. Next, set $x_0:=x^k_0$ and $x_{i+1} = f(x_i, R_{u_i^k}(\xi_i))$. Define $\bm{\eta}$ as the solution to $R_{x_i^k}(\eta_i)=x_i$. If we assume that $(\bfx, \bfu)$ satisfies (\ref{eq:constraints}), then we can write
\begin{subequations}\label{eq:perturbations}
\begin{align}
    &R_{x_{i+1}^k}(\eta_{i+1}) = f(R_{x_i^k}(\eta_i), R_{u_i^k}(\xi_i)), \label{eq:perturbed-dynamics} \\
    &s(R_{x_i^k}(\eta_i), R_{u_i^k}(\xi_i)) \leq 0 .\label{eq:perturbed-constraints}
\end{align} 
\end{subequations}

Next, we will take the 1st-order approximation of (\ref{eq:perturbations}). We require the following linear operators, 
\begin{subequations}\label{eq:differentials}
    \begin{align}
        \bfA_i^k &:= \diff_x f_{(x_i^k, u_i^k)}: T_{x_i^k} \calM \to T_{f(x_i^k, u_i^k)}\calM , \\ 
        \bfB_i^k &:= \diff_u f_{(x_i^k, u_i^k)}: T_{u_i^k}\calU \to T_{f(x_i^k, u_i^k)} \calM, \\
        \bfS_i^k &:= \diff_x s_{(x_i^k, u_i^k)}: T_{x_i^k} \calM \to \bbR^{n_s} , \\
        \bfQ_i^k &:= \diff_u s_{(x_i^k, u_i^k)}:  T_{u_i^k}\calU \to \bbR^{n_s}.
    \end{align}
\end{subequations}

\begin{prop}
    The 1st-order approximation of (\ref{eq:perturbations}) is
    \begin{subequations}\label{eq:geo-lin}
    \begin{align}
        &\eta_{i+1} = R^{-1}_{x_{i+1}^k}(z_i^k) + \bfD_i^k \circ (\bfA_i^k(\eta_i) + \bfB_i^k(\xi_i)), \label{eq:geo-lin-dynamics} \\
        &s(x_i^k, u_i^k) + \bfS_i^k(\eta_i) + \bfQ_i^k(\xi_i)\leq 0, \label{eq:geo-lin-constraints}
    \end{align}
    \end{subequations} where $z_i^k:= f(x_i^k,u_i^k)$ and
    \begin{equation}\label{eq:curvature-operator}
        \bfD_i^k := \diff (R^{-1}_{x_{i+1}^k})|_{z_i^k}:T_{z_i^k} \calM \to T_{x_{i+1}^k} \calM.
    \end{equation} 
\end{prop}
\begin{proof}
    From (\ref{eq:perturbed-dynamics}), we can write
    \begin{equation*}
        \eta_{i+1} = R_{x^k_{i+1}}^{-1}(f(R_{x_i^k}(\eta_i), R_{u_i^k}(\xi_i))).
    \end{equation*} This expression is a nonlinear mapping from $(\eta_i,\xi_i) \in T_{(x^k_i, u^k_i)} (\calM \times \calU)$ to $\eta_{i+1} \in T_{x^k_{i+1}} \calM$. So, we can compute its 1st-order approximation about $(0_{x^k_i},0_{u^k_i})$. Using the fact that $\diff R_{x_i^k}|_0=\text{id}$ and \begin{align*}
        \diff (R_{x^k_{i+1}}^{-1} \circ f \circ (R_{x_i^k}, R_{u_i^k})) = \diff R_{x^k_{i+1}}^{-1} \circ (\diff_x f + \diff_u f),
    \end{align*} through the chain rule, we obtain 
    \begin{align*}
        \eta_{i+1} = R_{x^k_{i+1}}^{-1}(z_i^k) + \tilde{\bfA}_i^k(\eta_i) + \tilde{\bfB}_i^k(\xi_i) + o(\|\eta_i\|_{x_i^k}, \|\xi_i\|_{u_i^k}).
    \end{align*} Here, $\tilde{\bfA}_i^k := \bfD_i^k \circ \bfA_i^k$ and $\tilde{\bfB}_i^k := \bfD_i^k \circ \bfB_i^k$. This gives us (\ref{eq:geo-lin-dynamics}); the inequality (\ref{eq:geo-lin-constraints}) follows by a similar procedure. 
\end{proof}

We remark that (\ref{eq:geo-lin}) describe linear constraints. Also, in the case $(\bfx, \bfu)$ is feasible, we have $R^{-1}_{x_{i+1}^k}(z_i^k) = 0$ and $\bfD_i^k=\text{id}$. We also have $\bfD_i^k = \text{id}$ in the scenario $\calM$ and $\calU$ are vector spaces.

\subsection{Slack Variables and Penalty}\label{subsec:slack-penalty}
Like SCvx, we use slack variables $s_i' \geq 0$ and $v_i \in T_{x_{i+1}^k}\calM$ to ensure that the sub-problem is feasible:
\begin{align*}
    s(x_i^k, u_i^k) + \bfS_i^k(\eta_i) + \bfQ_i^k(\xi_i) - s_i' \leq 0 \\
    \eta_{i+1} = R^{-1}_{x_{i+1}^k}(z_i^k) + \tilde{\bfA}_i^k(\eta_i) + \tilde{\bfB}_i^k(\xi_i) + v_i.
\end{align*}

For the sub-problem cost (\S\ref{subsec:linear-cost}), we require an exact penalty term in order to minimize usage of these slack variables. At time $k$, consider some function $P: T\calM \times \bbR^{n_s} \to \bbR$ that is defined only in coordinate-free quantities. From here, we can define the penalized cost as
\begin{equation}\label{eq:pen-cost}
    J(\bfx, \bfu) := C(\bfx, \bfu) + \sum_{i=0}^{N - 1} \lambda_i P (R^{-1}_{x_{i+1}^k}(z_i^k), s(x_i^k, u_i^k)).
\end{equation} 

The penalty employed in \S\ref{sec:example} for the unit quaternions is
\begin{equation}\label{eq:quat-penalty}
    P(v,s') = \|q^{-1} \cdot v\|_1 + \max(0,s'),
\end{equation} where $s' \in \bbR$ and $v \in T_q\calQ$.  

\subsection{Linearized Trajectory Cost}\label{subsec:linear-cost}
We will describe how to construct the linearized trajectory cost for the sub-problem. We will require that our stage and final stage costs are geodesically convex with respect to the Riemannian structures of $\calM$ and $\calU$. This will imply that the 2nd-order approximation is a positive semi-definite quadratic function, thereby making the sub-problem convex quadratic.  

A Riemannian metric is a smooth assignment of an inner product $\inner{.,.}_x: T_x \calM \times T_x \calM \to \bbR$ to each point $x \in \calM$. A smooth manifold paired with a Riemannian metric is called a Riemannian manifold. Riemannian metrics allow us to define generalized notions of the gradient.

Euclidean submanifolds equipped with $\inner{v,w}_x := v^{\top} w$ are labeled isometric. In the context of optimization, the ``right'' Riemannian metric can speed up the convergence rate of the corresponding Riemannian gradient method by orders of magnitude \cite{kraisler2024output}. 

The Riemannian gradient $\nabla f(x) \in T_x\calM$ of $f\colon \calM \to \bbR$ is the vector field that satisfies $\inner{\nabla f(x), v}_x = \diff f_x(v)$ at all $(x,v) \in T\calM$. While the Riemannian Hessian $\nabla^2 f(x):T_x\calM \to T_x\calM$ requires more machinery to define, for isometric Euclidean submanifolds it becomes
\begin{equation*}
    \nabla^2 f(x)[v] := \text{Proj}_{T_x\calM}\left(\overline{\nabla}^2 f(x)v \right).
\end{equation*} Here, $\overline{\nabla}^2 f(x) \in \bbR^{d \times d}$ is the ordinary Hessian matrix. We also define $\nabla^2 f(x)[v,w] := \inner{\nabla^2 f(x)[v], w}_x$.

We refer to~\cite{boumal2023introduction} on definitions of geodesic convexity for sets and functions on manifolds. There is one property of geodesically convex functions $f \colon U \subset \calM \to \bbR$ we focus on: $f$ is geodesically convex if and only if $\nabla^2 f\succeq0$. Then, the 2nd-order approximation about any $x \in U$ is a positive semi-definite quadratic:
\begin{equation*}
    \hat{f}_x(v):=f(x)+ \inner{v, \nabla f(x)}_x + \frac{1}{2}\nabla^2 f(x)[v,w].
\end{equation*} 

Suppose the stage costs $\phi(\cdot, \cdot),h(\cdot)$ are geodesically convex and defined on strongly geodesically convex domains $D \subset \calM$ and $U \subset \calU$. Then the trajectory cost (\ref{eq:traj-cost}) is strongly geodesically convex over $D^{N+1} \times U^N$. If we assume $\nabla_{xu} \phi=0$, then its 2nd-order approximation about any trajectory $(\bfx, \bfu)$ will be a positive semi-definite quadratic:
\begin{subequations}\label{eq:lin-cost}
\begin{align}
    \hat{C}_{(\bfx, \bfu)}(\bm{\eta},\bm{\xi}) = \sum_{i=1}^{N - 1} (\phi(x_i, u_i) + \inner{\nabla_x \phi(x_i, u_i), \eta_i}_{x_i} \\ 
    +\inner{\nabla_u \phi(x_i, u_i), \xi_i}_{u_i} + \frac{1}{2} \nabla_{xx}^2 \phi(x_i,u_i)[\eta_i, \eta_i] \\ 
    + \frac{1}{2} \nabla_{uu}^2 \phi(x_i,u_i)[\xi_i, \xi_i]) + h(x_N) \\ 
    + \inner{\nabla h(x_N), \eta_N}_{x_N} + \frac{1}{2}\nabla^2 h(x_N)[\eta_N, \eta_N] .
\end{align}
\end{subequations}

We now have the tools to define the linearized cost:
\begin{equation}\label{eq:pen-lin-cost}
    L^k(\bm{\eta}, \bm{\xi}; \bfv, \bfs') := \hat{C}_{(\bfx^k, \bfu^k)}(\bm{\eta}, \bm{\xi}) + \sum_{i=0}^{N - 1} \lambda_i P(v_i, s_i') .
\end{equation}

\subsection{Geodesic Local Sub-problem}\label{subsec:geo-local-subproblem}
The geodesic local sub-problem, in its abstract form, is
\begin{subequations}\label{eq:LOCP}
    \begin{align}
        \min_{\bm{\eta}, \bm{\xi}, \bfv, \bfs'} \ &L^k(\bm{\eta}, \bm{\xi}; \bfv, \bfs')\\
        \text{s.t. } &\eta_{i+1} = R_{x_{i+1}^k}^{-1}(z_i^k) + \tilde{\bfA}_i(\eta_i) + \tilde{\bfB}_i(\xi_i) + v_i, \\ 
        &s(x_i^k, u_i^k) + \bfS_i^k(\eta_i) + \bfQ_i^k(\xi_i) - s_i' \leq 0, \\
        &s_i' \geq 0,  \ \|\xi_i\|_{u_i^k} \leq r^k.
    \end{align}
\end{subequations}

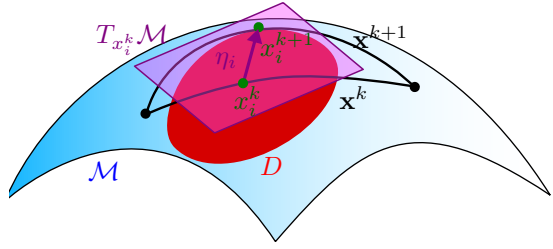
\begin{figure}
      \centering
      \begin{overpic}[width=.85\columnwidth]{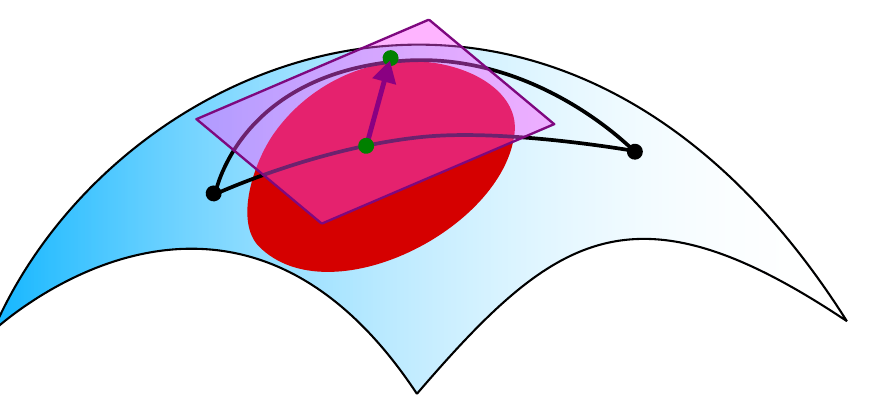}
          \put(30,28){\color{blue} $\calM$}
          \put(130,78){$\bfx^{k+1}$}
          \put(125,55){$\bfx^k$}
          \put(85,55){\color{green!40!black} $x^k_i$}
          \put(95,75){\color{green!40!black} $x^{k+1}_i$}
          \put(78,72){\color{violet} $\eta_i$}
        \put(33,80){\color{violet}$T_{x^k_i}\calM$}
          \put(95,30){\color{red} $D$} 
      \end{overpic}
      \caption{Visualization of perturbing a trajectory along a tangent space.}
      \label{fig:perturb}
  \end{figure}

\begin{lemma}
    Suppose the stage costs $\phi(\cdot,\cdot),h(\cdot)$ are twice continuously differentiable and geodesically convex over the strongly geodesically convex sets $U \times V$ and $V$, respectively. Then a solution exists for~(\ref{eq:LOCP}).
\end{lemma}
\begin{proof}
    Recall that the sub-problem is feasible if the feasibility set is non-empty. Due to the virtual control term $v_i$, any state in the feasible region of the convex sub-problem is reachable in finite time. Furthermore, the virtual buffer zone variables $s_i'$ ensure that the feasible set defined by the linearized state and control constraints is non-empty. Hence, an optimal solution exists. 
\end{proof}

Once solved, the trajectories are perturbed as follows: $\bfx^{k+1} = R_{\bfx^k}(\bm{\eta})$ and $\bfu^{k+1} = R_{\bfu^k}(\bm{\eta})$ (Figure \ref{fig:perturb}). In order to solve (\ref{eq:LOCP}), we must pick frames and compute the components of the tangent vectors and linear operators. We will show how to do this in a general set-up. 

Let $\calN$ be a Riemannian manifold of dimension $n$. Let $(E_i)$ be a local frame about $x \in \calN$. The components of a tangent vector $v \in T_x\calN$ is the unique coordinate vector $[v]:=(v^1,\ldots,v^n) \in \bbR^n$ for which $v = \sum_{i=1}^n v^i E_i(x)$. When $(E_i)$ is orthogonal, we can write $v^i = \inner{v, E_i(x)}_x$. Now, let $\calN'$ be another Riemannian manifold of dimension $n'$ and pick $x' \in \calN'$. Let $\bfA:T_x \calN \to T_{x'} \calN'$ be linear. Its components with respect to these frames will be a matrix $[\bfA] \in \bbR^{n' \times n}$. Set $A_j := \bfA(E_j(x)) \in T_{x'}\calM'$. Then $[A_j] \in \bbR^{n'}$ will be the $j$th column of $[\bfA]$. This also means $A_j = \sum_{i=1}^{n'}A_{ij}E_i'(x)$. If our frames are orthogonal, we have $A_{ij} = \inner{A_j, E'_i(x')}_{x'}$. Last, the components of the Riemannian metric will be a positive definite matrix $G_{ij}(x) = \inner{E_i(x),E_j(x)}_x$. For $w \in T_x\calM$, we can write $\inner{v,w}_x = [v]^\top [G(x)] [w]$. When $(E_i)$ is orthonormal, we have $[G]=I_n$. 

Using this methodology, we can see how to compute the necessary components in (\ref{eq:lin-cost}-\ref{eq:LOCP}) in order to realize (\ref{eq:LOCP}) and then solve using a convex program solver. For instance, $\nabla_{xx}^2 \phi[\eta,\eta] = [\eta]^\top[G][\nabla^2_{xx} \phi][\eta]$. For $\calU=\bbR^m$, an orthonormal global frame is the standard basis vectors. For $\calM=\calQ$, an orthonormal global frame is $E_i(q) := q \cdot e_i$, where $e_i \in \mathfrak{q} \cong \bbR^3$ is the $i$th standard basis vector.

\subsection{iSCvx Procedure}\label{subsec:algo}

We now introduce our algorithm, detailed in Algorithm \ref{alg:iscvx}. A few remarks. First, $J$ and $L$ are (\ref{eq:pen-lin-cost}) and (\ref{eq:lin-cost}), respectively. Second, rationale behind the algorithm parameters is laid out in \cite[Alg. 2.1]{mao2018successive}. Most importantly, since the algorithm, in particular Problem \ref{eq:LOCP}, relies solely on intrinsic quantities of the system manifolds, then the chosen coordinate frame has no impact on the computational outcome. Even more broadly, the outcome is invariant to diffeomorphism, provided the quantities used in the system and Problem \ref{eq:OCP} are related by a pullback with the corresponding quantities with the other representation. 

\begin{algorithm}[th]
\caption{iSCvx Algorithm}
\label{alg:iscvx}
\begin{algorithmic}[1]
    \State \textbf{Input:} Choose initial trajectory $(\bfx^1, \bfu^1) \in \calM^{N + 1} \times \calU^{N}$, $k \gets 1$. Select trust region $r^1 > 0$, penalty weight $\lambda > 0$, and parameters $0 < \rho_0 < \rho_1 < \rho_2 < 1$, $r_1 > 0$, and $\alpha > 1$, $\beta > 1$
    \While{$\Delta J^k > \epsilon_{tol}$ \textbf{do}}
        \State \textbf{step 1} Solve Problem (\ref{eq:LOCP}) to get $(\bm{\eta}, \bm{\xi}; \bfv, \bfs')$.
        \State \textbf{step 2} Compute 
        \begin{align*}
            \Delta J^k &= J(\bfx^k, \bfu^k) - J(R_{\bfx^k}(\bm{\eta}), R_{\bfu^k}(\bm{\xi})) \\
            \Delta L^k &= J(\bfx^k, \bfu^k) - L^k(\bm{\eta}, \bm{\xi}; \bfv, \bfs')
        \end{align*}
        \If{$\Delta J^k = 0$}
            \State \textbf{return} $(\bfx^k, \bfu^k)$;
        \Else
            \State Compute $\rho^k = \Delta J^k / \Delta L^k$.
        \EndIf
        \State \textbf{step 3} 
        \If{$\rho^k < \rho_0$}
            \State $r^k \gets r^k / \alpha$, go back to \textbf{step 1};
        \Else
            \State Update $\bfx^{k+1} \gets R_{\bfx^k}(\bm{\eta})$, $\bfu^{k+1} \gets R_{\bfu^k}(\bm{\xi})$
            \State Update $r^{k+1} \gets \begin{cases} 
                r^k / \alpha, & \text{if } \rho^k < \rho_1; \\
                r^k, & \text{if } \rho_1 \leq \rho^k < \rho_2; \\
                \beta r^k, & \text{if } \rho_2 \leq \rho^k.
            \end{cases}$
        \EndIf
        \State Update $r^{k+1} \gets \max\{r^{k+1}, r_1\}$, $k \gets k + 1$, go back to \textbf{step 1}.
    \EndWhile
    \State \textbf{return} $(\bfx^{k+1}, \bfu^{k+1})$.
\end{algorithmic}
\end{algorithm}

\section{Experiment: Constrained attitude guidance iSCvx} \label{sec:example}

\subsection{Set-up}
In this section, we will compare iSCvx to its ``Euclidean'' realization for constrained attitude guidance. Our presentation leverages the closed-form expressions for the quaternion exponential and logarithm, along with their derivatives \cite{sola2017quaternion, rossmann2006lie, hairer2006geometric}. 
$\diff \exp_\eta \colon \bbR^3 \to T_{\exp(\eta)} \calQ \subset \bbR^4$ and  $\diff \log_q : T_q \calQ \subset \bbR^4 \to \bbR^3$. These have the following closed form expressions:
\small
\begin{align*}
    \diff \exp_\omega(\eta) &= \mat{-\text{sinc}(\omega) \omega^{\top} \\ \text{sinc}(\omega) I_3  + \nabla \text{sinc}(\omega) \omega^{\top}} \eta \\
    \diff \log_q(v) &= \mat{\frac{-1}{\|q_v\|_2 \sqrt{1 - q_0^2}}q_v & \frac{\cos^{-1}(q_0)}{\|q_v\|_2}(I_3 - \frac{1}{\|q_v\|_2^2} q_vq_v^{\top} )} v \\ 
    \nabla \text{sinc}(\omega) &= \left(\frac{\cos\|\omega\|_2}{\|\omega\|_2^2} - \frac{\sin\|\omega\|_2}{\|\omega\|_2^3}\right)\omega. 
\end{align*} \normalsize Here, $(q,v) \in T\calQ$ and $\eta, \omega \in \bbR^3$. 

We fix an initial $q_0$ and desired $q_d$ attitude. Our retraction is (\ref{eq:lie-retraction}). The system is the discretized rotation kinematic equation,
\begin{equation*}
    q_{i+1} = f(q_i, \omega_i) := q_i \cdot \exp(\tau \omega_i),
\end{equation*} where $\tau>0$ is the timestep. This describes rotating an object with orientation $q_i$ about $\omega_i \in \bbR^3$ by an angle $\tau \|\omega_i\|_2$.

We set a boresight direction $y_b$, fixed in the body frame, and keep-out direction $t_o$, fixed in the inertial frame. The keep-out zone constraint is 
\begin{equation*}
    s(q) = t_o^\top (q \cdot y_b \cdot q^{-1}) - \cos(\theta_{\max}).
\end{equation*} So, $s(q) \leq 0$ implies the angle between $t_o$ and $y_o := q \cdot y_b \cdot q^{-1}$ is at least $\theta_{\max}$. For both SCvx and iSCvx, we use the penalty (\ref{eq:quat-penalty}).

For SCvx, we chose the following Euclidean trajectory cost
\begin{equation*}
    \sum_{i=0}^{N-1}\frac{1}{2}(\|q_i^k - q_d\|^2_2 + \|\omega_i^k\|^2_2) +\frac{1}{2}\|q_N^k - q_d\|_2^2 .
\end{equation*} The exact SCvx process is described in \cite[Alg. 2.1]{mao2018successive}. 

Now, we describe the iSCvx setup for the constrained attitude guidance problem. The squared geodesic distance $d_g(q,q_d)^2:= \|\log(q^{-1} \cdot q_d)\|_2^2$ is strongly geodesically convex over the domain $D:=\left\{q \in \calQ :d_g(q,q_d)< \pi/2 \right \}$ \cite{tron2012riemannian}. The stage costs are
\begin{align*}
    \phi(q_i^k, \omega_i^k) &=\frac{1}{2} \|\omega_i^k\|_2^2 + \frac{1}{2}d_g(q_i^k,q_d)^2, \\
    h(q_N^k) &= \frac{1}{2}d_g(q_N^k, q_d)^2.
\end{align*} Remark that this implies $\nabla_q \phi(q,\omega)=\nabla h(q)$, $\nabla_{q\omega}^2 \phi = 0$, and $\nabla_{qq}^2 \phi(q,\omega) = \nabla^2h(q)$. We designate (\ref{eq:traj-cost}) the \textit{geodesic} trajectory cost and restrict the domain to $D^{N + 1} \times \bbR^{3 \times (N - 1)}$.

The differentials of the system dynamics and constraints (\ref{eq:differentials}) can now be calculated as,
\small
\begin{align*}
    \bfA_i^k(\eta_i) &:= \diff_q f_{(q_i^k,\omega_i^k)}(\eta_i) = \eta_i \cdot \exp(\tau \omega_i^k),  \\ 
    \bfB_i^k(\xi_i) &:= \diff_\omega f_{(q_i^k,\omega_i^k)}(\xi_i) = \tau q_i^k \cdot \diff \exp_{\tau \omega_i^k}(\xi_i) , \\
    \bfS_i^k(\eta_i) &:= \diff s_{(q_i^k,\omega_i^k)}(\eta_i) = t_o^\top (q_i^k \cdot [(q_i^k)^{-1} \cdot \eta_i, y_b] \cdot (q_i^k)^{-1}) .
\end{align*} \normalsize
Set $z_i^k := f(q_i^k,\omega_i^k)$. Then for $v \in T_{z_i^k}\calQ$, (\ref{eq:curvature-operator}) becomes
\begin{equation*}
    \bfD_i^k(v) = q_{i+1}^k \cdot \diff \log_{(q_{i+1}^k)^{-1} \cdot z_i^k}((q_{i+1}^k)^{-1} \cdot v) .
\end{equation*}

The Riemannian gradient and Hessian of the stage costs becomes
\begin{subequations}
    \begin{align}
        \nabla h(q) &= q \cdot \log(q_d^{-1} \cdot q) \in T_q \calQ \\
        \nabla^2 h(q)[v] &= q \cdot \diff \log_{q_d^{-1} \cdot q}(q_d^{-1} \cdot v) \in T_q \calQ \\
        \nabla_\omega \phi(q,\omega) &= \omega \in T_\omega \mathfrak{q} \cong \bbR^3 \\
        \nabla_{\omega \omega}^2 \phi(q,\omega)[\xi] &= \xi \in T_\omega \mathfrak{q} \cong \bbR^3
    \end{align}
\end{subequations} We now have the tools to numerically solve Problem \ref{eq:LOCP}.


\subsection{Numerical Results} 

\begin{table}
\centering
\caption{$N=30$, $\tau=0.1$}
\begin{tabular}{|l|c|c|c|c|}
\hline
 & \multicolumn{2}{c|}{$\theta_{\max} = 10^\circ$} & \multicolumn{2}{c|}{$\theta_{\max} = 30^\circ$} \\ \cline{2-5}
 & \textbf{SCvx} & \textbf{iSCvx} & \textbf{SCvx} & \textbf{iSCvx} \\ \hline
\textbf{Ave. Iter. Count} & 40.21 & 24.89 & 45.8 & 26.8 \\ \hline
\textbf{Std. Iter. Count} & 9.23 & 2.14 & 18.28 & 1.88 \\ \hline
\textbf{Time} & 6.26 s & 4.40 s & 7.10 s & 4.70 s \\ \hline
\textbf{Ave Geo. Traj. Cost} & 4.98 & 4.73 & 6.50 & 5.67 \\ \hline
\textbf{Ave Eucl. Traj. Cost} & 7.21 & 6.91 & 9.65 & 8.17 \\ \hline
\end{tabular}
\label{table:data1}

\centering
\caption{$N=60$, $\tau=0.05$}
\begin{tabular}{|l|c|c|c|c|}
\hline
 & \multicolumn{2}{c|}{$\theta_{\max} = 10^\circ$} & \multicolumn{2}{c|}{$\theta_{\max} = 30^\circ$} \\ \cline{2-5}
 & \textbf{SCvx} & \textbf{iSCvx} & \textbf{SCvx} & \textbf{iSCvx} \\ \hline
\textbf{Ave. Iter. Count} & 67.9 & 24.75 & 65.72 & 25.65 \\ \hline
\textbf{Std. Iter. Count} & 34.86 & 2.22 & 17.02 & 2.45 \\ \hline
\textbf{Time} & 22.18 s & 9.09 s & 21.43 s & 9.37 s \\ \hline
\textbf{Ave Geo. Traj. Cost} & 9.04 & 8.96 & 10.59 & 10.50 \\ \hline
\textbf{Ave Eucl. Traj. Cost} & 11.94 & 13.34 & 13.98 & 15.42 \\ \hline
\end{tabular}
\label{table:data2}
\end{table}

\begin{figure}
    \centering
    \includegraphics[width=\columnwidth]{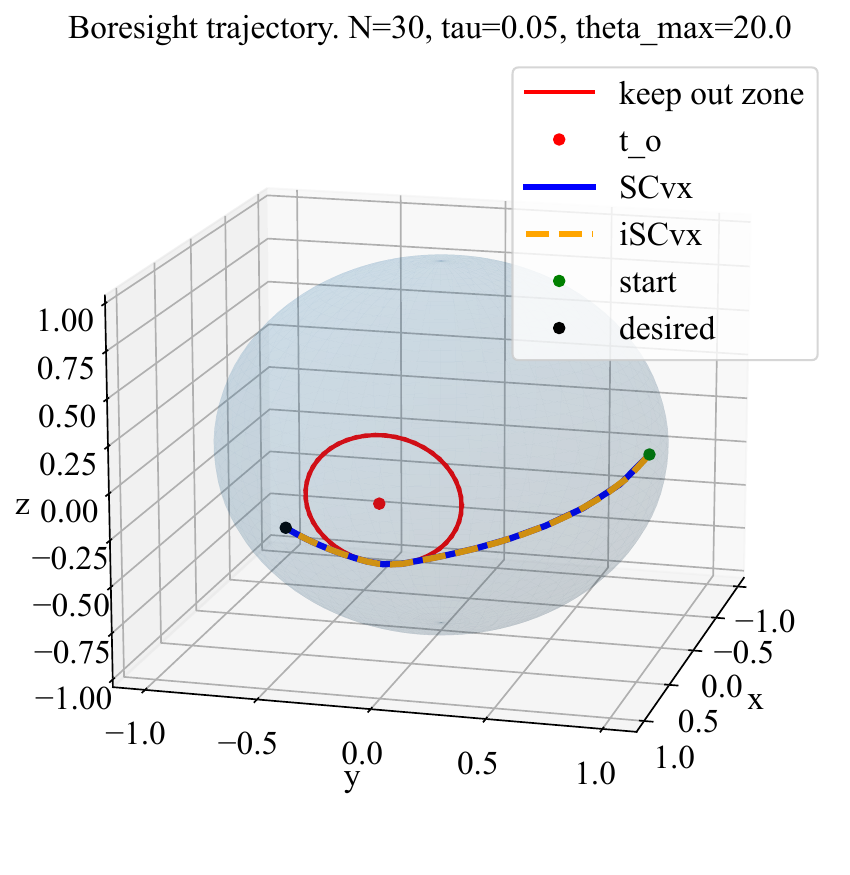}
    \caption{The iterates of SCvx and iSCvx for the constrained attitude guidance problem.}
    \label{fig:traj}
\end{figure}

The numerical simulation\footnote{Code available at github.com/Rainlabuw/intrinsic-scvx} is initiated with randomly chosen initial orientations $q_0$ and desired orientations $q_{d}$. The initial trajectory was derived through spherical linear interpolation. We implemented both SCvx and iSCvx with the termination criteria of $\epsilon_{\scriptstyle{\text{tol}}}=10^{-5}$. 
The remaining algorithm parameters are $r^1=1$, $r_l=0$, $\alpha=2$, $\beta=3.2$, $\rho_0=0$, $\rho_1=.25$, $\rho_2=.7$.

We repeated this process 100 times and recorded the average number of iterations for each algorithm. We performed the same numerical experimentation for the standard deviation of the number of iterations for each simulation, as well as the average clock time it took to solve the optimal control problem, and the average geodesic and Euclidean trajectory costs. The corresponding results are presented in Tables~\ref{table:data1} and \ref{table:data2}.

In table \ref{table:data1}, iSCvx outperforms SCvx in every category for $\theta_{\max}=30^\circ$, and all but the average Euclidean trajectory cost for $\theta_{\max}=10^\circ$. This is to be expected since iSCvx is not even attempting to optimize this cost. Paired with the fact that iSCvx achieves significantly lower average iteration count, standard deviation of iteration counts, and wall clock time, this experiment shows iSCvx has a natural advantage over SCvx for this particular optimal control problem.

The same behavior is observed in table \ref{table:data2}. Like before, while SCvx has a higher average Euclidean trajectory cost, the difference is small, especially considering iSCvx has significantly more attractive results in the other categories. Figure \ref{fig:traj} shows the optimal trajectories under both algorithms, and hence both costs. As shown, the trajectories are nearly identical.


\section{Conclusion}\label{conclusion}
This work has proposed the intrinsic successive convexification algorithm for solving non-convex trajectory optimization problems. We outlined the key ingredients of the proposed algorithm and discussed the computational benefits of the intrinsic geometry of the underlying smooth manifold during the iterates. A numerical example demonstrating the application of the so-called iSCvx to constrained attitude guidance is then discussed. Future directions include testing the algorithm on more complicated examples, namely powered descent, and generalizing this methodology to more general Sequential Convex Programming paradigms with local convergence guarantees via the coordinate-free Pontryagin maximum principle \cite[Ch. 12]{agrachev2013control}.

\addtolength{\textheight}{-12cm}   

\bibliographystyle{ieeetr}
\bibliography{references}

\end{document}